\documentclass{amsart}
\usepackage{color}
\usepackage{graphicx}
\usepackage{amsmath,amscd}
\usepackage{amssymb}
\input epsf
\usepackage{epsfig}

\usepackage{enumerate}
\newtheorem{theorem}{Theorem}[section]
\newtheorem{lemma}[theorem]{Lemma}
\newtheorem{corollary}[theorem]{Corollary}

\newtheorem{proposition}[theorem]{Proposition}
\theoremstyle{definition}
\newtheorem{definition}[theorem]{Definition}
\newtheorem{example}[theorem]{Example}

\theoremstyle{remark}
\newtheorem{remark}[theorem]{Remark}

\numberwithin{equation}{section}

\newcommand{\Real}{{\mathbb R}}

\newcommand{\eps}{\varepsilon}

\newcommand{\x}{\mathbf{x}}
\newcommand{\y}{\mathbf{y}}
\newcommand{\z}{\mathbf{z}}

\newcommand {\hide}[1]{}

\begin{document}
\title[Deciding whether a tropical linear prevariety is a tropical variety]
{Complexity of deciding whether a tropical linear prevariety is a tropical variety}
{
\let\thefootnote\relax\footnote{2010 Mathematics Subject Classification 14T05}
}
\author{Dima Grigoriev}
\address{CNRS, Math\'ematiques, Universit\'e de Lille, Villeneuve d'Ascq, 59655, France}
\email{dmitry.grigoryev@math.univ-lille1.fr}
\author{Nicolai Vorobjov}
\address{
Department of Computer Science, University of Bath, Bath
BA2 7AY, England, UK}
\email{nnv@cs.bath.ac.uk}

\begin{abstract}
We give an algorithm, with a singly exponential complexity,
deciding whether a tropical linear {\em prevariety} is a tropical linear {\em variety}.
The algorithm relies on a criterion to be a tropical linear variety in terms of a duality between the tropical
orthogonalization $A^\perp$ and the double tropical orthogonalization $A^{\perp \perp}$ of a subset $A$
of the vector space $(\Real \cup \{ \infty \})^n$.
We also give an example of a countable family of tropical hyperplanes such that their intersection
is not a tropical prevariety.
\end{abstract}

\maketitle

\section*{Introduction}

In this paper we use the operation $A^\perp$ of tropical orthogonalization, applied to a subset $A$ of
a vector space $(\Real \cup \{ \infty \})^n$, and its iteration, $A^{\perp \perp}$,
to formulate a criterion and an algorithm, deciding
whether a tropical linear {\em prevariety} is a tropical linear {\em variety}.

General concepts of tropical algebra can be found in \cite{BookJ, MS, RST}.
Specific questions of tropical linear algebra were considered in \cite{DSS, DS, S, YY}.

In Section~1 we list basic definitions, including the concept of a {\em tropical linear hull}
of a subset in $(\Real \cup \{ \infty \})^n$.
We recall a theorem, proved in \cite{BH, GK}, stating that any tropical linear prevariety
is the tropical linear hull of a finite set of vectors.
We list some properties of double orthogonalization,
in particular, we prove that $A^{\perp \perp}$ is the minimal tropical linear prevariety containing
the finite set $A$.
Further in Section~1, we recall a theorem, implicitly proved in \cite{S} (see also \cite{GK} and \cite{MT}),
stating that for two mutually complementary and orthogonal linear subspaces $P$ and $Q$
of the vector space $({\mathbb C}((t^{1/ \infty})))^n$ over Puiseux series, there exists a finite set
$A \subset (\Real \cup \{ \infty \})^n$ such that tropicalizations of $P$ and $Q$ coincide with
$A^\perp$ and $A^{\perp \perp}$ respectively.
This theorem is essential for the main result in Section~3, and
for reader's convenience we present its new short proof.

Sections 2 and 3 contain descriptions of some algorithms and estimates of their complexities.
Each of these algorithms can, in principle, be modelled by a Turing machine, and complexity is understood
in terms of its steps, i.e., as bit complexity.

In Section~2 we describe an algorithm, with singly exponential complexity in the number of variables, which for a given
(classical) algebraic set $V \subset ({\mathbb C}((t^{1/ \infty})))^n$ and a point $\x \in {\mathbb Q}^n$
decides whether or not $\x$ belongs to the tropicalization of $V$, and, if it does,
produces a lifting of $\x$ in $V$.
This algorithm serves as a key subroutine for algorithms in Section~3, but also is of independent interest.
An algorithm for the same problem follows from \cite{JMM}
and apparently has doubly exponential complexity bound in the number of variables.

Section~3 contains the main result of the paper: a criterion and deciding algorithms for
a tropical linear prevariety to be a tropical linear variety.
We propose two deciding algorithms.
One has the complexity exponential in bit-sizes of rational generators of the tropical linear hull representing
tropical linear prevariety and singly exponential in the number of variables.
Another algorithm is polynomial in bit-size of generators.
Its complexity also depends on the complexity of computing of a tropical basis (see \cite{BJS, HT, JS, KK}) of
a system of (classical) multivariate polynomial equations.
Apparently, the latter complexity is doubly exponential in the number of variables, in which case the complexity
of our second algorithm is also doubly exponential in the number of variables.

We also describe an algorithm which for a given tropical linear variety $A^\perp$
produces a linear subspace $P$ whose tropicalization coincides with $A^\perp$.
This algorithm has a singly exponential complexity in the number of variables.

In Section~4 we give an example of a countable family of tropical hyperplanes in
$(\Real \cup \{ \infty \})^6$ such that their intersection is not a tropical prevariety.
This strengthens examples in \cite{DW} (example of T.~Theobald) and \cite{GV} about countable intersections
of non-linear tropical hypersurfaces.

The extended abstract of this paper appeared in \cite{GV-CASC}.

\section{Preliminaries}

\subsection{Tropical linear prevarieties and tropical linear hulls}
We use the notation $\Real_\infty$ for $\Real \cup \{ \infty \}$.
We assume that for all $a \in \Real$ the rules $a < \infty$, $a+ \infty = \infty$, $\infty + \infty= \infty$,
and, for positive $a$, $a \cdot \infty = \infty$ hold.
The element $\infty$ is a ``tropical zero'', being the neutral element with respect to taking minimum.

\begin{definition}[\cite{DS}, \cite{MS}]\label{def:hyperplane}
For a given $(a_1, \ldots ,a_n) \in \Real_\infty^n$,
a {\em tropical hyperplane} in $\Real_\infty^n$ is the set of all points $(x_1, \ldots ,x_n) \in \Real_\infty^n$
at which the set $\{ x_1+a_1, \ldots , x_n+a_n \}$ has at least two minimal elements.
A {\em tropical linear prevariety} in $\Real_\infty^n$ is the intersection of a finite number of tropical hyperplanes.
\end{definition}

\begin{remark}
The point $(\infty, \ldots, \infty)$ belongs to every tropical linear prevariety.
A tropical hyperplane according to Definition~\ref{def:hyperplane} corresponds to the notion
of a codimension one linear subspace in classical linear algebra.
It can be identified with a special case, when $a_{n+1}= \infty$, of a more general notion of a tropical hyperplane,
defined as a set of all points $(x_1, \ldots ,x_n) \in \Real_\infty^n$ at which a set
$\{ x_1+a_1, \ldots , x_n+a_n ,a_{n+1}\}$, where $a_i \in \Real_\infty,\ 1 \le i \le n+1$, has at least two minimal elements.
\end{remark}

\begin{definition}[\cite{DS}, \cite{MS}]\label{def:perp}
Vectors ${\bf v}=(v_1, \ldots , v_n),\ {\bf a}=(a_1, \ldots ,a_n) \in \Real_\infty^n$ are called {\em tropically orthogonal}
if among numbers $v_i+a_i,\ 1 \le i \le n$ there are at least two minimal.
Note that $(\infty, \ldots, \infty)$ is tropically orthogonal to every vector ${\bf a} \in \Real_\infty^n$.
For a set of vectors $A= \{ {\bf a}_1, \ldots , {\bf a}_k  \} \subset \Real_\infty ^n$ denote by $A^{\perp}$
the set of all vectors in $\Real_\infty^n$ tropically orthogonal to each ${\bf a}_i,\ 1 \le i \le k$.
\end{definition}

It is clear that $\{ {\bf a} \}^\perp$ is a tropical hyperplane for a vector ${\bf a} \in \Real_\infty ^n$,
while $A^\perp$ is a tropical linear prevariety when $A \subset \Real_\infty ^n$ is finite.
Conversely, every tropical linear prevariety in $\Real_\infty ^n$ coincides with $A^\perp$ for a suitable finite set of vectors
$A \subset \Real_\infty ^n$.

\begin{definition}[\cite{DS}, \cite{MS}]
For a finite set of vectors $A= \{ {\bf a}_1, \ldots , {\bf a}_k  \} \subset \Real_{\infty}^n$ define its
{\em tropical linear hull} ${\rm Trophull} (A)$ as the set of all vectors in $\Real_{\infty}^n$ of the kind
$$\min_{1 \le i \le k}\{ t_i{\bf 1}_n+ {\bf a}_i \},$$
where $t_1, \ldots ,t_k$ are arbitrary elements in $\Real_{\infty}$,
$\min_{1 \le i \le k}$ denotes the component-wise minimum of a set of vectors, and ${\bf 1}_n=(1, \ldots ,1)$ is
the unit vector in $\Real^n$.
For an arbitrary subset $X \subset \Real_\infty ^n$ define ${\rm Trophull}(X)$ as the union of sets
${\rm Trophull}(A)$ over all finite subsets $A \subset X$.
\end{definition}

Note that ${\rm Trophull} (A)$ is the direct tropical analogy of the concept of {\em linear hull} of
a finite set of vectors in classical linear algebra.
It always contains the point $(\infty, \ldots ,\infty) \in \Real_\infty^n$,
because all $t_i$, $1 \le i \le k$ can be chosen to be $\infty$.




\begin{definition}\label{def:chart1}
For every partition $\{ i_1, \ldots ,i_p \} \cup \{ i_{p+1}, \ldots ,i_n \}$ of $\{ 1, \ldots ,n \}$
a {\em chart} in $\Real_\infty ^n$ is an open convex polyhedron
$$
C_{i_1, \ldots ,i_p}:=\{ x_{i_1}= \cdots =x_{i_p}= \infty \} \cap \{ x_{i_{p+1}}< \infty \}
\cap \cdots \cap \{x_{i_n} < \infty \} \subset \Real_\infty ^n.
$$
\end{definition}
Clearly, $\Real_\infty ^n$ is the union of all $2^n$ pair-wise disjoint charts.

One can extend the standard concepts of a convex polyhedron and a finite polyhedral complex to the
case of the subsets of the space $\Real_\infty ^n$ (see \cite{GK}).
Restriction of a convex polyhedron $P \subset \Real_\infty ^n$ to a chart $C_{i_1, \ldots ,i_p}$
coincides with a usual convex polyhedron in $\Real^{n-p}$ translated by a vector in $\{0, \infty \}^{n-p}$
with $\infty$ in positions $i_1, \ldots ,i_p$.
Hence, $P$ is a finite union of translated usual convex polyhedra, and we define the dimension
$\dim (P)$ as the maximum of the dimensions of restrictions of $P$ to all charts.
The dimension of a finite polyhedral complex is defined as the maximum of dimensions of its convex polyhedra.

The following theorem directly follows from \cite[Theorem~1]{GK} (part (2)
of the theorem, except the complexity bound, was proved earlier in \cite[Proposition~2]{BH}).

Let $A= \{ {\bf a}_1, \ldots , {\bf a}_k  \} \subset \Real_\infty ^n$ be a set of vectors.

\begin{theorem}\label{th:main}
\begin{enumerate}
\item
The set ${\rm Trophull} (A)$ is a union of all convex polyhedra of a polyhedral complex in $\Real_\infty ^n$.
\item
For the tropical linear prevariety $A^\perp \subset \Real_\infty ^n$
there exists a finite set of vectors $\{ {\bf b}_1, \ldots , {\bf b}_N \} \subset \Real_\infty ^n$ such that
$A^\perp={\rm Trophull} ( \{ {\bf b}_1, \ldots ,{\bf b}_N \})$.
Moreover, there is an algorithm which for a given set $\{ {\bf a}_1, \ldots , {\bf a}_k  \}$
of vectors in ${\mathbb Q}_\infty ^n$, with bit-sizes of coordinates not exceeding $L$, computes the set
$\{ {\bf b}_1, \ldots , {\bf b}_N \} \subset {\mathbb Q}_\infty ^n$.
The complexity of this algorithm is polynomial in $L$ and $n^k$.
Bit-sizes of coordinates of the computed vectors ${\bf b}_1, \ldots , {\bf b}_N$ do not exceed
$L+ \log k$, while $N=O(n^k)$.
\end{enumerate}
\end{theorem}

\begin{corollary}\label{cor:main}
Every tropical linear prevariety $A^\perp \subset \Real_\infty ^n$ is
the union of all convex polyhedra of a polyhedral complex in $\Real_\infty ^n$.
\end{corollary}





Theorem~\ref{th:main}, (2) states that any tropical linear prevariety
$A^\perp= \{ {\bf a}_1, \ldots ,{\bf a}_n \}^\perp \subset \Real_\infty ^n$
coincides with the tropical linear hull of a finite subset of its vectors.
This is not necessarily true for the restriction $A^\perp \cap \Real^n$ (see Example~\ref{ex:A_0} below).

\subsection{Dual tropical linear prevarieties}
We extend the operation $X^\perp$, introduced in Definition~\ref{def:perp}, so that it can be applied to arbitrary
(not necessarily finite) subsets $X \subset \Real_\infty ^n$.
Namely, denote by $X^\perp$ the set of all vectors in $\Real_\infty ^n$ orthogonal to each ${\bf a} \in X$.
We will use notations $X^{\perp \perp}:= (X^\perp)^\perp$ and $X^{\perp \perp \perp}:= (X^{\perp \perp})^\perp$.

\begin{remark}
Observe that by the definition, for a finite subset $A \subset \Real_\infty ^n$, the set
$A^{\perp \perp}$ is an intersection of an infinite number of tropical hyperplanes in $\Real_\infty ^n$.
As we will show in Section~\ref{sec:infinite} below, not every intersection of even countable number
of tropical hyperplanes is a union of cells of a finite polyhedral complex, let alone tropical linear prevariety.
However, in the special case of a finite $A$, the set $A^{\perp \perp}$
is a tropical linear prevariety (Proposition~\ref{prop:min_trop}).
\end{remark}

\begin{lemma}\label{le:inclusions}
For any subset $X \subset \Real_\infty ^n$ we have:
\begin{enumerate}
\item
${\rm Trophull}(X) \subset X^{\perp \perp}$;
\item
$X^\perp = X^{\perp \perp \perp}$.
\end{enumerate}
\end{lemma}

\begin{proof}
\noindent (1)\ Directly follows from definitions.
\medskip

\noindent (2)\ Inclusions $X \subset X^{\perp \perp}$ and $X^\perp \subset X^{\perp \perp \perp}$ are trivial.
The first of these inclusions implies that every $\x \in X^{\perp \perp \perp}$ is orthogonal to every $\y \in X$.
Hence $\x \in X^\perp$.
\end{proof}


\begin{proposition}\label{prop:min_trop}
Let $A$ be a finite set of vectors in $\Real_\infty ^n$.
Then $A^{\perp \perp}$ is the minimal (with respect to the subset relation) tropical linear prevariety containing $A$.
\end{proposition}

\begin{proof}
By Theorem~\ref{th:main}, $A^\perp= {\rm Trophull} (\{ {\bf b}_1, \ldots {\bf b}_N \})$ for some
${\bf b}_1, \ldots {\bf b}_N \in \Real_\infty ^n$, hence
$A^{\perp \perp}=( {\rm Trophull} (\{ {\bf b}_1, \ldots {\bf b}_N \}))^\perp$.
According to \cite{DSS, DS},
$$( {\rm Trophull} (\{ {\bf b}_1, \ldots {\bf b}_N \}))^\perp= \{ {\bf b}_1, \ldots {\bf b}_N \}^\perp,$$
which implies that $A^{\perp \perp}$ is a tropical linear prevariety.



Let $C$ be any tropical linear prevariety containing $A$.
Then $C= \{ {\bf c}_1, \ldots , {\bf c}_M \}^\perp$ for some vectors
${\bf c}_1, \ldots , {\bf c}_M \in \Real_\infty ^n$.
Since ${\bf c}_1, \ldots , {\bf c}_M \in A^\perp$, we get the inclusion
$A^{\perp \perp} \subset \{ {\bf c}_1, \ldots , {\bf c}_M \}^\perp=C$.
Because $A \subset A^{\perp \perp}$ (cf. proof of Lemma~\ref{le:inclusions} (2)), we conclude that
$A^{\perp \perp}$ is the minimal tropical linear prevariety containing $A$.
\end{proof}

Lemma~12 in \cite{GP} (also Theorem~4.2 in \cite{DSS}) implies that $\dim (A^\perp) + \dim (A^{\perp \perp}) \ge n$
for finite $A \subset \Real_\infty ^n$.
The following example shows that the equality  $\dim (A^\perp) = \dim (A^{\perp \perp})=n-1$
is possible for some $A \subset \Real^n$, which does not happen in classical linear algebra.

\begin{example}[cf. \cite{GP}]\label{ex:A_0}
Let $A_0= \{ {\bf a}_1, \ldots ,{\bf a}_{n-1} \} \subset \Real^n$, where
$${\bf a}_i=(\underbrace{1, \ldots ,1,0}_{i},1, \ldots ,1,0,0)\> \text{for}\>
1 \le i \le n-2\> \text{and}\> {\bf a}_{n-1}=(1, \ldots ,1,0,0).$$
It is easy to see that every vector $\x=(x_1, \ldots ,x_n) \in A_0^\perp \subset \Real_\infty ^n$
should have minimal elements $x_{n-1}, x_n$, and, conversely, every vector $\x$ with minimal elements $x_{n-1}, x_n$
is in $A_0^\perp$.
Therefore,
$$A_0^\perp= \{ t {\bf 1}_n +(c_1, \ldots ,c_{n-2},0,0)|\> \text{for all}\> 0 \le c_i \in \Real_\infty\> \text{and}\>
t \in \Real_\infty \}.$$
\end{example}

From the example, we see that every vector
$\y=(y_1, \ldots ,y_n) \in A_0^{\perp \perp} \subset \Real_\infty ^n$
should have minimal elements $y_{n-1}, y_n$, and, conversely, every vector $\y$ with minimal elements $y_{n-1}, y_n$
is in $A_0^{\perp \perp}$.
It follows that $A_0^{\perp \perp}=A_0^{\perp}$.
Note that $\dim (A_0^{\perp})=\dim (A_0^{\perp \perp})=n-1$.

For another property of this example, observe that
$$A_0^{\perp}= {\rm Trophull}(\{ {\bf b}_1, \ldots ,{\bf b}_{n-1} \}),$$
where
$${\bf b}_i=(\underbrace{\infty, \ldots ,\infty,0}_{i},\infty, \ldots ,\infty,0,0)\> \text{for}\>
1 \le i \le n-2\> \text{and}\> {\bf b}_{n-1}=(\infty, \ldots ,\infty,0,0).$$

On the other hand, take $n=3$ in the example.
For any finite set $\{ {\bf v}_1, \ldots ,{\bf v}_N \} \subset A_0^\perp \cap \Real^3$ we have
$$A_0^\perp \cap \Real^3 \neq {\rm Trophull}(\{ {\bf v}_1, \ldots ,{\bf v}_N \}) \cap \Real^3.$$
Indeed, assume, for contradiction, that there is a set of vectors $V= \{ {\bf v}_1, \ldots ,{\bf v}_N \}$ such that
$A_0^\perp \cap \Real^3 = {\rm Trophull}(V) \cap \Real^3$.
Without loss of generality, assume that ${\bf v}_j=(c_j,0,0)$ for each $1 \le j \le N$, where $c_j \ge 0$.
Denoting $M:= \max_{1 \le j \le N} \{ c_j \}$, we consider the vector $(M+1,0,0) \in A_0^\perp \cap \Real^3$
and prove that it is not in ${\rm Trophull} (V)$.
Indeed, if
\begin{equation}\label{eq:contradiction}
(M+1,0,0)= \min_{1 \le j \le N} (t_j {\bf 1}_3 + {\bf v}_j)
\end{equation}
for some $t_1, \ldots ,t_N$, then, considering either second or third coordinate,
we conclude that every $t_j \ge 0$ and at least one of them, $t_{j_0}=0$.
It follows that the first coordinate of the vector $\min_{1 \le j \le N} (t_j {\bf 1}_3 + {\bf v}_j)$
in (\ref{eq:contradiction}) is at most $c_{j_0}$, which is less than $M+1$.


\subsection{Tropicalization of linear subspaces}
Let $K$ be an algebraically closed field of characteristic 0 and $\mathbb F$ denote the field
$K((t^{1/\infty}))$ of Puiseux series over $K$.
For an element $y \in {\mathbb F}$ different from 0, let ${\rm val}(y) \in {\mathbb Q}$ denote
the {\em valuation} of the element $y$ in ${\mathbb F}$, i.e., the power in the lowest term of the Puiseux series $y$.
Separately define ${\rm val}(0)= \infty$.

\begin{definition}[cf. \cite{MS}]\label{def:trop_hyper}
Let $f:=a_1x_1+ \cdots + a_nx_n$, where $0 \neq a_i \in {\mathbb F}$ for all $1 \le i \le n$.
The {\em formal tropicalization} of the hyperplane $\{ f=0 \} \subset {\mathbb F}^n$ is the tropical hyperplane,
${\rm Tropf} (\{ f=0 \}) \subset \Real^n$, defined by the set
$\{ y_1+ {\rm val} (a_1), \ldots ,y_n+ {\rm val} (a_n) \}$ (see Definition~\ref{def:hyperplane}).
\end{definition}

By Kapranov's Theorem \cite[Theorem~3.1.3]{MS}, ${\rm Tropf} (\{ f=0 \})$ coincides with the (Euclidean)
closure in $\Real^n$ of the countable set
$$\{ ({\rm val}(x_1), \ldots ,{\rm val}(x_n))|\> (x_1, \ldots ,x_n) \in \{f=0 \} \cap ({\mathbb F} \setminus \{ 0 \})^n \}.$$

The following definition is ``dual'' to Definition~\ref{def:chart1}.

\begin{definition}\label{def:chart2}
For every partition $\{ i_1, \ldots ,i_r \} \cup \{ i_{r+1}, \ldots ,i_n \}$ of $\{ 1, \ldots ,n \}$
a {\em chart in} ${\mathbb F}^n$ is a set
$$
D_{i_1, \ldots ,i_r}:=\{ x_{i_1}= \cdots =x_{i_r}= 0 \} \cap \{ x_{i_{r+1}} \neq 0 \}
\cap \cdots \cap \{x_{i_n} \neq 0 \} \subset {\mathbb F}^n.
$$
\end{definition}

For any $X \subset {\mathbb F}^n$ we obviously have
$$X = \bigcup_{\{ i_1, \ldots ,i_r \}} (X \cap D_{i_1, \ldots ,i_r}),$$
where the union is taken over all subsets $\{i_1, \ldots ,i_r \}$ of $\{ 1, \ldots ,n \}$.

\begin{definition}[cf. \cite{MS}]\label{def:tropicalization}
The {\em tropicalization} ${\rm Trop} (X \cap D_{i_1, \ldots ,i_r})$ of $X \cap D_{i_1, \ldots ,i_r}$
is the set of all points $(y_1, \ldots ,y_n) \in \Real_\infty ^n$ such that $y_{i_1}= \cdots =y_{i_r}= \infty$
and $(y_{i_{r+1}}, \ldots ,y_{i_n})$ belongs to the Euclidean closure in $\Real^{n-r}$ of the set
$$\{ ({\rm val}(x_{i_{r+1}}), \ldots ,{\rm val}(x_{i_n}))|\> (x_{1}, \ldots ,x_{n})
\in X \cap D_{i_1, \ldots ,i_r} \}.$$
The {\em tropicalization} ${\rm Trop} (X)$ of $X$ is defined as
$\bigcup_{ \{ i_1, \ldots ,i_r \}} {\rm Trop}(X \cap D_{i_1, \ldots ,i_r})$.
\end{definition}

\begin{remark}\label{re:subset_formal}
Definition~\ref{def:tropicalization} immediately implies that for any two sets $X, Y \subset {\mathbb F}^n$ there is
the inclusion ${\rm Trop}(X \cap Y) \subset {\rm Trop} (X) \cap {\rm Trop} (Y)$.
The inverse inclusion $\supset$ is not generally true even for linear subspaces.
\end{remark}

Let $P \subset {\mathbb F}^n$ be a linear subspace, with
$\dim (P)=d$, and $\z_1, \ldots ,\z_d \in P$ be a basis of $P$.
Recall that {\em Pl\"ucker coordinates} of $P$ in the Grassmanian ${\rm Gr}(d, {\mathbb F}^n)$ are all
$(d \times d)$-minors $p_{j_1, \ldots ,j_d}$ of the matrix with rows
$\z_1, \ldots ,\z_d$, corresponding to the columns $1 \le j_1 < \cdots < j_d \le n$.
Any $\z=(z_1, \ldots ,z_n) \in P$ satisfies the relation
\begin{equation}\label{eq:pluecker}
\sum_{1 \le i \le d+1} (-1)^i p_{j_1, \ldots ,j_{i-1}, j_{i+1}, \ldots ,j_{d+1}} z_{j_i}=0,
\end{equation}
for every subset of columns $1 \le j_1 < \cdots < j_{d+1} \le n$.
Note that the relations in (\ref{eq:pluecker}) are independent of the choice of a basis in $P$.

Denote the set of points $\z$ satisfying (\ref{eq:pluecker}) by $P_{j_1, \ldots ,j_{d+1}}$.

The following statement is a strengthening for $\Real_\infty$ of \cite[Proposition~4.2]{S}
(also \cite[Theorem~4.3.17]{MS}), and can be proved analogously.

\begin{lemma}\label{le:speyer}
\begin{equation}\label{eq:trop_subset}
{\rm Trop} (P) = \bigcap_{j_1, \ldots ,j_{d+1}}{\rm Trop} (P_{j_1, \ldots ,j_{d+1}}).
\end{equation}
\end{lemma}

Let $Q \subset {\mathbb F}^n$ be a linear subspace orthogonal to $P$ with $\dim (Q)=n-d$.
According to \cite{BJS, S}, the tropicalizations ${\rm Trop} (P)$ and ${\rm Trop} (Q)$ are tropically orthogonal, with
$\dim ({\rm Trop}(P))=d$ and $\dim ({\rm Trop}(Q))=n-d$.

The following theorem implicitly appears in \cite{S}.
Formulated in the language of valuated matroids, an analogous statement can be found in \cite{MT}.
To make our exposition closed, for reader's convenience, we give another proof, using the technique of
Pl\"ucker relations.

\begin{theorem}\label{th:trop_orth}
There is a finite subset $A \subset \Real_\infty ^n$ such that $ {\rm Trop} (P)=A^\perp$
and ${\rm Trop} (Q)= A^{\perp \perp}$.
\end{theorem}

\begin{proof}
Consider the relations (\ref{eq:pluecker}) for the subspace $P$.
For every subset $1 \le j_1 < \cdots < j_{d+1} \le n$
introduce a vector ${\bf p}_{j_1, \ldots ,j_{d+1}} \in {\mathbb F}^n$ such that its $j_i$-coordinate is
$(-1)^i p_{j_1, \ldots ,j_{i-1}, j_{i+1}, \ldots ,j_{d+1}}$
for $1 \le i \le d+1$, and all the rest of coordinates are equal to $0$.
Then every vector in $P_{j_1, \ldots ,j_{d+1}}$ is orthogonal to the vector ${\bf p}_{j_1, \ldots ,j_{d+1}}$, hence
every vector in $\bigcap_{ \{j_1, \ldots ,j_{d+1} \}}P_{j_1, \ldots ,j_{d+1}}$ is orthogonal to
vectors ${\bf p}_{j_1, \ldots ,j_{d+1}}$ for all subsets $1 \le j_1 < \cdots < j_{d+1} \le n$.
By ${\rm val} ({\bf p}_{j_1, \ldots ,j_{d+1}})$ denote the vector in $\Real_\infty ^n$ obtained from
${\bf p}_{j_1, \ldots ,j_{d+1}}$ by taking ${\rm val} (\cdot)$ of all non-zero coordinates and replacing
all zero coordinates by $\infty$.
As $A$ take the set of all vectors ${\rm val} ({\bf p}_{j_1, \ldots ,j_{d+1}}) \in \Real_\infty ^n$
for all $1 \le j_1 < \cdots < j_{d+1} \le n$.
By Definition~\ref{def:trop_hyper},
$$A^\perp =\bigcap_{1 \le j_1 < \cdots < j_d \le n} \bigcup_I{\rm Tropf} (P_{j_1, \ldots ,j_{d+1}} \cap D_I),$$
where the union is taken over all $I \subset \{ 1, \dots ,n \}$ and $D_I$
are the corresponding charts in ${\mathbb F}^n$.
By Kapranov's Theorem,
${\rm Tropf} (P_{j_1, \ldots ,j_{d+1}} \cap D_I)={\rm Trop} (P_{j_1, \ldots ,j_{d+1}} \cap D_I)$,
while, by Definition~\ref{def:tropicalization},
$\bigcup_I{\rm Trop} (P_{j_1, \ldots ,j_{d+1}} \cap D_I)= {\rm Trop} (P_{j_1, \ldots ,j_{d+1}})$.
Then, by Lemma~\ref{le:speyer}, $A^\perp = {\rm Trop}(P)$.

Since the linear subspaces $P$ and $Q$ are orthogonal, for every $\y \in Q$ we have $P \subset \{ \y \cdot \z=0 \}$.
Tropicalizing both sides of this inclusion, we conclude that ${\rm Trop} (Q)$ and ${\rm Trop} (P)$ are tropically
orthogonal.
Hence, ${\rm Trop} (Q) \subset({\rm Trop} (P))^\perp=A^{\perp \perp}$.
Vectors ${\bf p}_{j_1, \ldots ,j_{d+1}}$ lie in $Q$, since they are orthogonal to $P$.
It follows that $A \subset {\rm Trop} (Q)$.
By Proposition~\ref{prop:min_trop}, $A^{\perp \perp} \subset {\rm Trop} (Q)$, since ${\rm Trop} (Q)$ is
a tropical linear prevariety.
We conclude that $A^{\perp \perp}={\rm Trop} (Q)$.
\end{proof}

\begin{corollary}\label{cor:orthog}
Let $P, Q \subset {\mathbb F}^n$ be orthogonal complements  of one another and ${\rm Trop} (P)= A^\perp$.
Then ${\rm Trop} (Q)=A^{\perp \perp}$.
\end{corollary}

\begin{proof}
By Theorem~\ref{th:trop_orth}, there is a finite subset $B \subset \Real_\infty ^n$ such that
${\rm Trop} (P)= B^\perp$ and ${\rm Trop} (Q)= B^{\perp \perp}$.
Since ${\rm Trop} (P)= A^\perp$, we get $A^\perp=B^\perp$, hence ${\rm Trop} (Q)= A^{\perp \perp}$.
\end{proof}

\begin{remark}
Corollary~\ref{cor:orthog} implies, in particular, that in Example~\ref{ex:A_0} the tropical linear prevariety $A_0^\perp$
is not a tropical linear variety.
\end{remark}

\section{Testing membership in a tropical variety and computing a lifting}

For computational purposes, from now on we will assume that the field $K= \overline{\mathbb Q}$, thus
${\mathbb F}:= \overline{\mathbb Q}((t^{1/ \infty}))$ is the field of formal Puiseux series in
a variable $t$ with complex algebraic coefficients.

\begin{definition}\label{def:lifting}
Let $\x \in ({\mathbb Q} \cup \{\infty \})^n$ and ${\bf v} \in {\mathbb F}^n$ such that ${\rm Trop} ({\bf v})=\x$.
Then ${\bf v}$ is called a {\em lifting} of $\x$.
\end{definition}

In this section we describe an algorithm for testing membership of a given point $\x \in {\mathbb Q}^n$ in a tropical variety.
If $\x$ does belong to the variety, the algorithm computes its lifting.
The complexity of the algorithm is singly exponential in $n$.

For elements $\alpha_1, \ldots , \alpha_m \in {\mathbb F}$ that are algebraic over ${\mathbb Q}(t)$ let $\eta$
be the primitive element of the algebraic extension
${\mathbb Q}(t)(\alpha_1, \ldots ,\alpha_m)$ of the field ${\mathbb Q}(t)$, hence
${\mathbb Q}(t)(\alpha_1, \ldots ,\alpha_m)={\mathbb Q}(t)[\eta]$ (see \cite[Theorem~4.6]{Leng}).
The primitive element $\eta$ is defined (up to a conjugacy class) by its minimal (irreducible) polynomial
$\psi \in {\mathbb Z}[t][Z]$, i.e., $\psi (\eta)=0$.
Then $\alpha_i= \zeta_i (\eta)$ for a polynomial $\zeta_i \in {\mathbb Q}(t)[Z]$, for every $1 \le i \le m$.

Consider (classical) polynomials $f_1, \ldots, f_k \in {\mathbb F}[X_1, \ldots ,X_n]$, where the set of all coefficients
$\alpha_1, \ldots, \alpha_m \in {\mathbb F}$ of $f_1, \ldots, f_k$ is represented, as above, via a primitive element.
Let $V(f_1, \ldots ,f_k):= \{ f_1= \cdots =f_k=0 \} \subset {\mathbb F}^n$.
We describe an algorithm, with singly exponential complexity, which for given $f_1, \ldots ,f_k$
and a point $\x=(x_1, \ldots ,x_n) \in {\mathbb Q}^n$
decides whether or not $\x$ belongs to ${\rm Trop}(V(f_1, \ldots ,f_k)) \subset \Real^n$.
Moreover, if $\x \in {\rm Trop}(V(f_1, \ldots ,f_k))$ the algorithm produces a lifting of $\x$
(see Definition~\ref{def:lifting}).
According to \cite{G15}, this computational problem is NP-hard.

Without loss of generality, we assume that $\x= {\bf 0}$, performing if necessary
the coordinate transformation $X_i \to X_i t^{-x_i}$ in polynomials $f_j$, $1 \le j \le k$.
We keep the same notation, $f_j$, $1 \le j \le k$, for polynomials appearing after the change of coordinates.

The rest of the algorithm consists of three parts.
Firstly, it reduces testing whether $\bf 0$ belongs to ${\rm Trop}(V(f_1, \ldots ,f_k))$ to testing whether
$\bf 0$ lies in the tropicalization of one of 0-dimensional algebraic sets.
Next, the algorithm does the latter testing.
Finally, in the case of positive result, it constructs the lifting of $\bf 0$.

\subsection{Reduction to 0-dimensional algebraic sets.}
Applying algorithms from \cite{C86, G86}, represent $V(f_1, \ldots ,f_k)$ as the finite union of its
irreducible over ${\mathbb Q}(t)$ components, $V(f_1, \ldots ,f_k)= \bigcup_\nu V_\nu$.
Each irreducible component is represented in two forms: by a system of equations with coefficients in ${\mathbb Q}(t)$,
and by a generic point (a point such that the field generated by its coordinates has the transcendence degree,
over the field generated by the coefficients of the equations, equal to the dimension of the component).
More precisely, a generic point of an irreducible component $V_\nu$, with $\dim (V_\nu)=s$, is represented as follows.
Among coordinates $X_1, \ldots ,X_n$ a transcendence basis is chosen, let it be $X_1, \ldots ,X_s$.
Then coordinates $X_{s+1}, \ldots ,X_n$ are algebraic over ${\mathbb Q}(X_1, \ldots, X_s)(t)$, and
algorithms from \cite{C86, G86} describe the primitive element $\mu$ of the extension
${\mathbb Q}(X_1, \ldots, X_n)(t)$ as $\mu = \lambda_{s+1} X_{s+1} + \cdots + \lambda_n X_n $ for positive
integers $\lambda_{s+1}, \ldots ,\lambda_n$, and by its minimal polynomial $\chi \in {\mathbb Q}(X_1, \ldots ,X_s)(t)[Z]$.
Each coordinate $X_i$, $s+1 \le i \le n$ is then represented by a polynomial expression in $\mu$ with coefficients
in ${\mathbb Q}(X_1, \ldots ,X_s)(t)$.

Let $Y_0, \ldots ,Y_n$ be new variables.
Fix a component $V_\nu$ with $\dim (V_\nu)>0$, and consider the algebraic set
$\widetilde V_\nu$ defined by the same equations as $V_\nu$ over the field
$\overline{\mathbb Q}(Y_0,Y_1, \ldots ,Y_n)((t^{1/ \infty}))$.
Let
$$\widetilde W_\nu:= \widetilde V_\nu \cap \left\{ Y_0+ \sum_{1 \le i \le n}Y_iX_i=0 \right\}.$$
Then \cite[Proposition~4.6]{JMM} (see also \cite[Theorem~2.7.8]{OP}, \cite[Theorem~3.5]{JY}) implies
that if ${\bf 0} \in {\rm Trop} (V_\nu)$, then ${\bf 0} \in {\rm Trop}(\widetilde W_\nu)$.
The converse statement is obvious.

Let $\dim (V_\nu)=s$.
Introducing new sets of variables, $Y_{p,q}, 0 \le q \le n$, where $1 \le p \le s$, and adding equations
$Y_{p,0}+ \sum_{1 \le i \le n}Y_{p,i}X_i=0$ to the system of equations defining $\widetilde V_\nu$,
we will obtain a system of equations defining a zero-dimensional algebraic set $\widetilde U_\nu$ over the field
$\overline{\mathbb Q}(\{ Y_{p,q} \}_{1 \le p \le s, 0 \le q \le n})((t^{1/ \infty}))$.
Find all points in $\widetilde U_\nu$ using algorithms from \cite{C86, G86}.
The algorithm will conclude that ${\bf 0} \in {\rm Trop}(V_\nu)$ if and only if
for some point ${\bf v} \in \widetilde U_\nu$ its tropicalization is ${\bf 0}$.
We now describe how the algorithm checks the latter condition.

\subsection{Testing existence of the lifting of a point in 0-dimensional algebraic set.}
Algorithms from \cite{C86, G86} represent each point in ${\bf v}=(v_1, \ldots ,v_n) \in \widetilde U_\nu$ as follows.
For every $1 \le i \le n$ the coordinate $v_i$ is represented as $v_i= \xi_i ( \theta)$,
where $\theta$ is a root of an irreducible polynomial $\varphi \in {\mathbb Z}(\{ Y_{p,q} \}_{p, q })(t)[Z]$,
with $\deg_Z(\varphi)= r$, while
the expression $\xi_i ( \theta)$ is of the form
$$
\xi_i ( \theta)= \frac{1}{b_i} \sum_{0 \le \ell < r}c_{i, \ell}\ \theta^\ell,
$$
where $b_i,\ c_{i, \ell} \in {\mathbb Z}[\{ Y_{p,q} \}_{p, q }](t)$.
We can represent $\xi_i ( \theta)$ in a form $\xi_i (\theta)= a_i + R(t)$, where $R(t)=o(1)$ is the series
of terms containing $t$ in positive degree, while $a_i$ is the sum of terms of degrees in $t$ at most $0$.

More precisely, let $\alpha$ be the smallest integral (possibly negative) degree in $t$ of the rational in $t$
coefficients in the expression $\xi_i$.
Let $\beta$ be the degree of the first term in the Puiseux series $\theta=\theta(t)$.
For a positive integer $\gamma \ge \beta$, let $\Theta$ be the sum of all consecutive terms in the Puiseux series
$\theta$ starting with the term of degree $\beta$ and ending with the term of the smallest degree not less than $\gamma$.

If $\beta > 0$, then choose $\gamma$ to be the smallest integer such that $\alpha + \gamma \ge 0$.
Then $\xi_i (\Theta)$ contains all terms in $\xi_i(\theta)$ having non-positive degrees in $t$.

If $\beta \le 0$, then $\Theta^{r-1}$ contains the term with the degree not less than $ \alpha +(r-2) \beta +\gamma$,
hence choosing $\gamma$ to be the smallest integer such that $\alpha +(r-2) \beta + \gamma \ge 0$,
we guarantee that $\xi_i (\Theta)$ contains all terms in $\xi_i(\theta)$ having non-positive degrees in $t$.

Using the algorithm from \cite{C86-2}, find sum $\Theta$ of all first consecutive terms in
the Puiseux expansion of $\theta$ in $t$ up to the term of the smallest degree not less than $\gamma$.
Substituting $\Theta$ into $\xi_i$, we compute the sum $a_i$ of terms in $\xi_i(\theta)$ of degrees in $t$ at most $0$.
Let for each $i$, among the terms in $a_i$, the lowest degree in $t$ be $0$, i.e., $a_i$ is independent of $t$.
Then we conclude that ${\bf 0}$ belongs to ${\rm Trop}(V(f_1, \ldots ,f_k)) \subset \Real^n$.
If among the terms in $a_i$ the lowest degree in $t$ is different from $0$ for some $i$,
then ${\rm Trop} ({\bf v}) \neq {\bf 0}$.
If the latter condition holds for every point ${\bf v} \in \widetilde U_\nu$, then ${\bf 0} \not\in {\rm Trop}(V_\nu)$.

\subsection{Constructing the lifting of a point.}
To find a lifting of $\bf 0$ (in case of $a_i$ independent of $t$ for all $i$), observe that for each $i$
the element $a_i$ belongs to a finite extension of the field ${\mathbb Q}(\{ Y_{p,q} \}_{p, q })$, and is defined by
a polynomial expression via the primitive element of this extension, while the primitive element
is defined by its minimal polynomial (cf. representation of algebraic elements
$\alpha_1, \ldots \alpha_m$ at the beginning of this section).
All coefficients in the polynomial expressions and the minimal polynomial belong to
${\mathbb Q}( \{ Y_{p,q} \}_{p, q })$.
Choose a point $\y \in {\mathbb Z}^{s(n+1)}$ such that
\begin{enumerate}
\item
after the substitution $\y$ in variables $\{ Y_{p,q} \}_{p, q }$ neither of the denominators of these coefficients vanish;
\item
after the substitution $\y$ in variables $\{ Y_{p,q} \}_{p, q }$ neither of the
numerators at leading terms of minimal polynomials vanish.
\end{enumerate}
Coordinates of $\bf y$ can be taken to be non-negative integers not exceeding the degrees of
the polynomials mentioned in (1), (2).
Thus, one can estimate the complexity of choosing point $\bf y$ by the bound on the degrees of these polynomials
which is provided below.

The condition (2) implies that after the substitution of $\y$ in $(Y_1, \ldots, Y_{s(n+1)})$
each minimal polynomial may become reducible but its degree does not change.

Recall that each coordinate $v_i$ in ${\bf v}=(v_1, \ldots ,v_n) \in \widetilde U_\nu$ is represented
by polynomials $\varphi,\ \xi_i \in {\mathbb Q}(\{ Y_{p,q} \}_{p, q })(t)[Z]$.
Substituting in coefficients of these polynomials the vector of variables $(Y_1, \ldots, Y_{s(n+1)})$
by $\y$, we obtain a representation of a point in $(\overline{{\mathbb Q}(t)})^n$ which is a lifting of $\bf 0$.
This ends the description of the algorithm.


\subsection{Complexity of the algorithm.}
Let $d=\max_{1 \le j \le k} \deg_{X_1, \ldots ,X_n} (f_j)$ and $\delta$ be the maximum among the degrees in $t$
of polynomials $\psi, \zeta_i$, which define coefficients $\alpha_i$ of polynomials $f_j$.
Let $M$ be the maximum among the common denominator of rational numbers $x_1, \ldots , x_n$ and
the absolute values of their numerators.
By the Bezout Theorem, the degree of the algebraic set $V:=V(f_1, \ldots ,f_k)$ does not exceed $D:=d^n$.
After the change of coordinates $X_i \to X_i t^{-x_i}$ in polynomials $f_j$, $1 \le j \le k$, the bound
$\delta$ will be multiplied a positive number not exceeding $M^2$.

Algorithms in \cite{C86, G86}, applied to intersections of the irreducible components of $V$
(considered over the field $\overline{\mathbb Q}(\{ Y_{p,q} \}_{p, q })((t^{1/ \infty}))$) with
$s$ hyperplanes
$$\left\{ Y_{p,0}+ \sum_{1 \le i \le n}Y_{p,i}X_i=0 \right\},$$
produce a union of $0$-dimensional irreducible (over ${\mathbb Q}(\{ Y_{p,q} \}_{p,q})(t)$) algebraic sets
(conjugacy classes) each defined by some polynomials
$$g_1, \ldots, g_r \in {\mathbb Q}(\{ Y_{p,q} \}_{p,q})(t)[X_1, \ldots ,X_n],$$
where $r \le D^n$, of degrees at most $D^{n^2}$ with respect to $O(n^2)$ variables $\{ Y_{p,q} \}_{p,q}$,
and with coefficients represented by univariate polynomials $\varphi,\ \xi_i$ of degrees at most $D^2$
with respect to $Z$ and at most $\delta M^2 D^{n^2}$ with respect to $t$.
Complexity of the algorithm from \cite{C86-2}, applied to polynomials $\varphi,\ \xi_i$, is polynomial in
$\delta M D^{n^2}$ and bit-sizes of their rational coefficients.

As a result, the complexity of the algorithm is polynomial in bit-sizes of rational coefficients of
polynomials $\psi, \zeta_i$, which define coefficients $\alpha_i$ of the input polynomials $f_j$,
in the number $k$ of these polynomials, in $\delta M$, and in $D^{n^4}= d^{n^5}$.

We now summarise the results of this section.

\begin{theorem}\label{th:lifting}
There is an algorithm which for given polynomials $f_1, \ldots, f_k \in {\mathbb F}[X_1, \ldots ,X_n]$ and
a point $\x=(x_1, \ldots ,x_n) \in {\mathbb Q}^n$
decides whether or not $\x$ belongs to ${\rm Trop}(V(f_1, \ldots ,f_k)) \subset \Real^n$.
Moreover, if $\x \in {\rm Trop}(V(f_1, \ldots ,f_k))$ the algorithm produces a lifting of $\x$.

To describe the complexity, let $d=\max_{1 \le j \le k} \deg_{X_1, \ldots ,X_n} (f_j)$,
let $\delta$ be the maximum among the degrees in $t$
of polynomials $\psi, \zeta_i$, which define coefficients $\alpha_i$ of polynomials $f_j$.
Let $M$ be the maximum among the common denominator of rational numbers $x_1, \ldots , x_n$ and
absolute values of their numerators.
Then the complexity of the algorithm is polynomial in bit-sizes of rational coefficients of
$\psi, \zeta_i$, in the number $k$ of polynomials $f_j$, in $\delta M$, and in $d^{n^5}$.
\end{theorem}

\begin{remark}
The part of algorithm in the proof of Theorem~\ref{th:lifting}, which checks whether $\bf 0$ is
a tropicalization of a point in $\widetilde U_\nu$, actually computes tropicalization of the (0-dimensional)
$\widetilde U_\nu$.
Since the dimension of a (classical) algebraic set in ${\mathbb F}^n$ coincides with the dimension
of its tropicalization, the same procedure can be used to compute the tropicalization of any 0-dimensional
algebraic set.
\end{remark}

\begin{remark}
In case of tropical {\em linear} varieties $(d=1)$, a deciding algorithm, with complexity similar to the one
in Theorem~\ref{th:lifting}, is given in \cite{G15}.
\end{remark}

\section{Criterion and deciding algorithm for being a tropical linear variety}

Let $A \subset ({\mathbb Q} \cup \{\infty \})^n$ be a set of $k$ vectors with bit-sizes of coordinates not exceeding $L$.
Since $A^\perp$ and $A^{\perp \perp}$ are tropical linear prevarieties (see Proposition~\ref{prop:min_trop}),
Theorem~\ref{th:main} implies that
$$A^\perp={\rm Trophull}(\{ \x_1, \ldots ,\x_p \})\quad \text{and}\quad
A^{\perp \perp}={\rm Trophull}(\{ \y_1, \ldots ,\y_q \})$$
for some vectors $\x_1, \ldots ,\x_p, \y_1, \ldots ,\y_q \in ({\mathbb Q} \cup \{\infty \})^n$.
Moreover, the algorithm, stated in Theorem~\ref{th:main}, computes vectors $\x_1, \ldots ,\x_p$,
with complexity polynomial in $L, n^k$, and vectors $\y_1, \ldots ,\y_q$
with complexity polynomial in $L, n^{n^k}$.

\begin{theorem}\label{th:criterion}
The following three statements are equivalent.
\begin{enumerate}
\item
There exist mutually complementary and orthogonal linear subspaces $P, Q$ in ${\mathbb F}^n$ such that
$A^\perp ={\rm Trop}(P)$, and $A^{\perp \perp} ={\rm Trop} (Q)$
(in particular, $A^\perp$ , $A^{\perp \perp}$ are tropical linear {\em varieties}).
\item
There exist liftings
$${\bf v}_1, \ldots {\bf v}_p, {\bf w}_1, \ldots ,{\bf w}_q \in {\mathbb F}^n\quad \text{of vectors}\quad
\x_1, \ldots \x_p, \y_1, \ldots ,\y_q \in ({\mathbb Q} \cup \{\infty \})^n$$
respectively, such that $({\bf v}_i, {\bf w}_j)=0$ for all $1 \le i \le p,\> 1 \le j \le q$.
\item
$A^\perp$ is a tropical linear {\em variety}.
\end{enumerate}
\end{theorem}


For the proof of this theorem we recall the following definition.

\begin{definition}[\cite{DSS}]\label{def:trop_sing}
A square $(r \times r)$-matrix with elements $m_{ij} \in \Real_\infty,\ 1 \le i,j \le r$, is called
{\em tropically singular} if the minimum among
$$
\{ m_{1 \sigma_1}+ \cdots +m_{r \sigma_r} |\>
\text{for all permutations}\> (\sigma_1, \ldots ,\sigma_r)\> \text{of}\> (1, \ldots, r)\}
$$
is attained on at least two elements, and {\em tropically non-singular} otherwise.
The {\em tropical rank}, ${\rm trk}(M)$, of a matrix $M$ is the largest integer $r$ such that this matrix has a non-singular
$(r \times r)$-submatrix.
\end{definition}

\begin{proof}[Proof of Theorem~\ref{th:criterion}]
Implication $(1) \Rightarrow (2)$.
Suppose there exist subspaces $P, Q \subset {\mathbb F}^n$ as in (1).
Then arbitrary liftings ${\bf v}_1, \ldots ,{\bf v}_p \in P$ and ${\bf w}_1, \ldots ,{\bf w}_q \in Q$,
of $\x_1, \ldots ,\x_p$ and $\y_1, \ldots ,\y_q$ respectively, satisfy (2).

We now prove implication $(2) \Rightarrow (1)$.
Let $P$ (respectively, $Q$) be the linear hull of ${\bf v}_1, \ldots ,{\bf v}_p$ (respectively, of
${\bf w}_1, \ldots ,{\bf w}_q$).
Since $({\bf v}_i, {\bf w}_j)=0$ for all $1 \le i \le p,\> 1 \le j \le q$, each vector ${\bf v} \in P$ is orthogonal
to each vector ${\bf w}_j$, in particular, subspaces  $P$ and $Q$ are orthogonal.
Hence, ${\rm Trop} ({\bf v}) \perp \y_j$ for each $1 \le j \le q$.
Then ${\rm Trop} ({\bf v}) \perp \y$ for each
$\y \in {\rm Trophull} (\{ \y_1, \ldots ,\y_q \})= A^{\perp \perp}$,
thus ${\rm Trop} ({\bf v}) \in A^{\perp \perp \perp}$.
By Lemma~\ref{le:inclusions}, (2), this means ${\rm Trop} ({\bf v}) \in A^\perp$, thus ${\rm Trop}(P) \subset A^\perp$.
Similarly, we can prove that ${\rm Trop}(Q) \subset A^{\perp \perp}$.

We prove the inclusion $A^{\perp} \subset {\rm Trop}(P)$ by restricting it to charts in $\Real^n_\infty$
(see Definition~\ref{def:chart1}).
Choose a chart $C_I$ in $\Real^n_\infty$, where $I \subset \{ 1, \ldots ,n \}$, in $\Real^n$, and the corresponding
chart $D_I$ in ${\mathbb F}^n$ (see Definition~\ref{def:chart2}).

Take $\x \in A^{\perp} \cap C_I$.

Consider first the case when $\x=(x_1, \ldots ,x_n)$ is representable, via generators of $A^{\perp}$, as
$\x= \min_{1 \le i \le p}  \{ r_i{\bf 1}_n +\x_i \}$ with $r_i \in {\mathbb Q}$, where the minimum is taken component-wise.
Recall that $P$ is the linear hull of ${\bf v}_1, \ldots ,{\bf v}_p \in {\mathbb F}^n$,
let ${\bf v}_i=(v_{i1}, \ldots ,v_{in})$.
For each $1 \le j \le n$ let $\alpha_{1j}, \ldots ,\alpha_{pj}$ be coefficients of lowest degrees in Puiseux series
$v_{1j}t^{r_1}, \ldots ,v_{pj}t^{r_p}$ respectively.
Choose $s_1, \ldots ,s_p \in {\mathbb C}$ so that $s_1 \alpha_{1j}+ \cdots + s_p \alpha_{pj} \neq 0$ for each $1 \le j \le n$.
Then
$${\rm Trop} (s_1 v_{1j}t^{r_1}+ \cdots +s_p v_{pj}t^{r_p})= \min \{ x_{1j}+r_1, \ldots ,x_{pj}+r_p \}= x_j,$$
i.e., $\x=(x_1, \ldots ,x_n) \in {\rm Trop} (P \cap D_I)$.

In case when in the representation $\x= \min_{1 \le i \le p}  \{ r_i{\bf 1}_n +\x_i \}$
elements $r_i$ belong to $\Real$, consider a sequence of points
$\left\{\min_{1 \le i \le p}  \{ r^{(\ell)}_i{\bf 1}_n +\x_i \} \right\}_{1 \le \ell < \infty}$
in $A^\perp \cap C_I$ converging to $\x$, such that each $r^{(\ell)} \in {\mathbb Q}$.
Then $\x \in {\rm Trop} (P \cap D_I)$, because ${\rm Trop} (P \cap D_I)$ is closed.

Taking the union over all subsets $I \subset \{ 1, \ldots ,n \}$, in the inclusions
$A^\perp \cap C_I \subset {\rm Trop}(P \cap D_I)$ we get $A^\perp \subset {\rm Trop}(P)$.

Similarly, we can prove that $A^{\perp \perp} \subset {\rm Trop}(Q)$.

We conclude that $A^{\perp} = {\rm Trop}(P)$ and $A^{\perp \perp} = {\rm Trop}(Q)$.

We now prove that subspaces $P$ and $Q$ are mutually complementary.
In \cite{GP} it is shown that $\dim (A^\perp) + \dim (A^{\perp \perp}) \ge n$.
Hence, by \cite[Theorem~4.2]{DSS},
$${\rm trk}(\{ \x_1, \ldots ,\x_p \}) + {\rm trk}(\{ \y_1, \ldots ,\y_q \})= \dim (A^\perp)
+ \dim (A^{\perp \perp}) \ge n.$$
On the other hand, \cite{DSS} implies that
$${\rm rk}(\{ {\bf v}_1, \ldots ,{\bf v}_p \}) \ge
{\rm trk}( \{ \x_1, \ldots ,\x_p \})\quad \text{and}\quad {\rm rk}(\{ {\bf w}_1, \ldots ,{\bf w}_p \}) \ge
{\rm trk}( \{ \y_1, \ldots ,\y_q \}).$$
It follows that
${\rm rk}(\{ {\bf v}_1, \ldots ,{\bf v}_p \}) + {\rm rk}(\{ {\bf w}_1, \ldots ,{\bf w}_p \}) \ge n$,
hence $P$ and $Q$ are mutually complementary, given that they are orthogonal.

Finally, implication $(1) \Rightarrow (3)$ is trivial.
To prove $(3) \Rightarrow (1)$ notice that since there is a subspace $P \subset {\mathbb F}^n$ such that
${\rm Trop} (P)=A^\perp$, by Corollary~\ref{cor:orthog}, ${\rm Trop} (Q)=A^{\perp \perp}$ for the
subspace $Q$ complement orthogonal to $P$.
\end{proof}

\begin{corollary}\label{cor:deciding}
There is an algorithm which for given tropical linear {\em prevarieties}
$$A^\perp= {\rm Trophull} (\{ \x_1, \ldots ,\x_p \})\> \text{and}\
A^{\perp \perp}= {\rm Trophull} (\{ \y_1, \ldots ,\y_q \}),$$
where $\x_1, \ldots ,\x_p, \y_1, \ldots ,\y_q \in ({\mathbb Q} \cup \{ \infty \})^n$,
decides whether $A^\perp$ is a tropical linear {\em variety}.
The complexity of the algorithm is exponential in bit-sizes of rational coordinates of vectors $\x_i, \y_j$,
$1 \le i \le p,\ 1 \le j \le q$, and in $n,p,q$.

\end{corollary}

\begin{proof}
The input of the algorithm under construction is the set
$$\{ \x_1, \ldots ,\x_p, \y_1, \ldots ,\y_q \} \subset ({\mathbb Q} \cup \{ \infty \})^n.$$

Consider, over ${\mathbb F}^\ast \cong {\mathbb F} \setminus \{0 \}$, the system of equations
\begin{equation}\label{eq:dot_product}
\sum_{1 \le \nu \le n} V_{i \nu} W_{j \nu}=0
\end{equation}
for all $1 \le i \le p$, $1 \le j \le q$ such that
$x_{i \nu} \neq \infty$ and $y_{j \nu} \neq \infty$ for all $1 \le \nu \le n$,
where $(V_{i1}, \ldots ,V_{in}), (W_{j1}, \ldots ,W_{jn})$ are vectors of variables.

Consider vectors $\x_i$, $\y_j$, $1 \le i \le p$, $1 \le j \le q$ in the input,
from which coordinates $x_{i \nu}$ and $y_{j \nu}$ with either $x_{i \nu}= \infty$ or $y_{j \nu}= \infty$ are removed.
Applying Theorem~\ref{th:lifting}, the algorithm checks whether there exist liftings of these vectors,
satisfying the system of equations (\ref{eq:dot_product}).
If no, then $A^\perp$ is not a tropical linear variety, by Theorem~\ref{th:criterion}.
Otherwise, let $P$ be the linear hull of vectors ${\bf v}_1, \ldots {\bf v}_p$ such that in every ${\bf v}_i$
each coordinate corresponding to $x_{i \nu} \neq \infty$ is the lifting $v_{i \nu}$, while
each coordinate corresponding to $x_{i \nu}= \infty$ is $0$.
Similarly, let $Q$ be the linear hull of vectors ${\bf w}_1, \ldots {\bf w}_q$ such that in every ${\bf w}_j$
each coordinate corresponding to $y_{j \nu} \neq \infty$ is the lifting $w_{j \nu}$, while
each coordinate corresponding to $y_{j \nu} = \infty$ is $0$.
Then, by Theorem~\ref{th:criterion}, $P$ and $Q$ are mutually complementary and orthogonal linear subspaces
of ${\mathbb F}^n$, while $A^\perp ={\rm Trop}(P)$ and $A^{\perp \perp} ={\rm Trop} (Q)$.
In particular, $A^\perp$ and $A^{\perp \perp}$ are tropical linear varieties.

Theorem~\ref{th:lifting} implies that the complexity of the algorithm is exponential in bit-sizes
of rational coordinates of vectors $\x_i$, $1 \le i \le p$, $\y_j$, $1 \le j \le q$, and in $n,p,q$.
\end{proof}

\begin{remark}
There is an alternative algorithm for the problem described in Corollary~\ref{cor:deciding}.
Using \cite{BJS, HT} or \cite{KK}, it constructs a tropical basis
of the system of equations (\ref{eq:dot_product}), which is a finite set of polynomials $H_\ell$,
with integer coefficients and $(p+q)n$
variables $V_{i \nu}, W_{j \nu}$ such that $x_{i \nu} \neq \infty$ and $y_{j \nu} \neq \infty$.
(See a detailed definition and properties of a tropical basis in \cite[Section~2.6]{MS}.)
The algorithm checks whether, for all $H_\ell$, vectors $\x_1, \ldots \x_p, \y_1, \ldots ,\y_q$,
from which coordinates $x_{i \nu}$ and $y_{j \nu}$ with either $x_{i \nu}= \infty$ or $y_{j \nu}= \infty$
are removed, satisfy tropicalizations ${\rm Trop}(H_\ell)$.
If yes, then, by the definition of tropical basis, there exist liftings $v_{i \nu}, w_{j \nu} \in {\mathbb F}^\ast$
of all $x_{i \nu} \neq \infty, y_{j \nu} \neq \infty$ respectively, which satisfy the system (\ref{eq:dot_product}).
Then the algorithm continues as the algorithm from Corollary~\ref{cor:deciding}.
If vectors $\x_1, \ldots \x_p, \y_1, \ldots ,\y_q$ from which coordinates $x_{i \nu}$ and $y_{j \nu}$ with either
$x_{i \nu}= \infty$ or $y_{j \nu}= \infty$ are removed, do not satisfy ${\rm Trop}(H_\ell)$ for all $H_\ell$,
then $A^\perp$ is not a tropical linear variety, by Theorem~\ref{th:criterion}.

The complexity of this algorithm is polynomial in the bit-size of rational coordinates of vectors $\x_i$,
$1 \le i \le p$, $\y_j$, $1 \le j \le q$.
The complexity also depends on the complexity of computing of a tropical basis for (\ref{eq:dot_product}).
Apparently, the latter complexity is doubly exponential in $n, p, q$, though we are unaware of a proof
of this bound in literature.
Note that there is a doubly exponential upper bound on the {\em degree} of a tropical basis
\cite{JS} (cf. also \cite{MR, GRS}).
\end{remark}

\begin{corollary}
There is an algorithm which for tropical linear {\em prevarieties}
$$A^\perp= {\rm Trophull} (\{ \x_1, \ldots ,\x_p \})\> \text{and}\>
A^{\perp \perp}= {\rm Trophull} (\{ \y_1, \ldots ,\y_q \}),$$
where $\x_1, \ldots ,\x_p, \y_1, \ldots ,\y_q \in ({\mathbb Q} \cup \{ \infty \})^n$,
produces bases of linear subspaces $P$ and $Q$, such that $A^\perp= {\rm Trop} (P)$ and $A^{\perp \perp}= {\rm Trop} (Q)$,
in case these subspaces exist.
The complexity of the algorithm is exponential in bit-sizes of rational coordinates of vectors $\x_i$,
$1 \le i \le p$, $\y_j$, $1 \le j \le q$, and in $n,p,q$.
\end{corollary}

\begin{proof}
Apply the algorithm from Theorem~\ref{th:lifting} to vectors $\x_i$, $\y_j$, $1 \le i \le p$, $1 \le j \le q$,
from which coordinates $x_{i \nu}$ and $y_{j \nu}$ with either $x_{i \nu}= \infty$ or $y_{j \nu}= \infty$ are removed,
and to the system (\ref{eq:dot_product}).
The algorithm from Theorem~\ref{th:lifting} will either produce liftings $v_{i \nu}, w_{j \nu} \in {\mathbb F}^\ast$
of all $x_{i \nu} \neq \infty, y_{j \nu} \neq \infty$ respectively, which satisfy (\ref{eq:dot_product}), or will
indicate that liftings do not exist, i.e., vectors $\x_i, \y_j$ with removed coordinates do not belong to
the tropicalization of (\ref{eq:dot_product}).
If liftings exist, then $P$ is the linear hull of vectors ${\bf v}_1, \ldots {\bf v}_p$ such that in every ${\bf v}_i$
each coordinate corresponding to $x_{i \nu} \neq \infty$ is the lifting $v_{i \nu}$ of $x_{i \nu}$, while
each coordinate corresponding to $x_{i \nu}= \infty$ is $0$.
Similarly, $Q$ is the linear hull of vectors ${\bf w}_1, \ldots {\bf w}_q$ such that in every ${\bf w}_j$
each coordinate corresponding to $y_{j \nu} \neq \infty$ is the lifting $w_{j \nu}$ of $y_{j \nu}$, while
each coordinate corresponding to $y_{j \nu} = \infty$ is $0$.
Herewith, all coordinates of vectors ${\bf v}_i, {\bf w}_j$ are Puiseux series in
$\overline{\mathbb Q}((t^{1/\infty}))$ represented as algebraic elements over the field ${\mathbb Q}(t)$,
as described in the proof of Theorem~\ref{th:lifting}.
Choosing maximal linearly independent vectors among ${\bf v}_1, \ldots {\bf v}_p$ (respectively, among
${\bf w}_1, \ldots {\bf w}_q$), we get a basis of $P$ (respectively, of $Q$).

Theorem~\ref{th:lifting} implies that the complexity of this algorithm is
exponential in the bit-size of rational coordinates of vectors $\x_i$,
$1 \le i \le p$, $\y_j$, $1 \le j \le q$, and exponential in $n, p, q$.
\end{proof}

\section{Infinite intersections of tropical linear prevarieties}\label{sec:infinite}

By the definition, the intersection of a finite number of tropical hyperplanes is a tropical linear prevariety.
In this section we give an example of a {\em countable} family of tropical hyperplanes in $\Real_\infty ^6$
such that their intersection is not a finite union of convex polyhedra, in particular, not a tropical prevariety.
This strengthens examples in \cite{DW} (example of T.~Theobald) and \cite{GV},
in which intersections of a countable families of tropical
(non-linear) prevarieties were shown not to be finite unions of convex polyhedra.

Choose a sequence $\{ \eps_i \}_{i=1}^\infty$ of pair-wise distinct real numbers $\eps_i$ such that $0< \eps_i <1/4$, and
consider the tropical hyperplane $L_i \subset \Real_\infty ^6$ defined by the set
$${\mathcal L}_i(x_1,x_2,y_1,y_2,z_1,z_2):=\{ -i+x_1, -i+x_2, -i/2- \eps_i +y_1, -i/2+y_2, z_1, z_2 \}$$
(see Definition~\ref{def:hyperplane}).

Let
$$M:=\bigcap_{1 \le i <\infty}L_i  \subset \Real_\infty ^6.$$

\begin{proposition}
The set $M \cap \Real^6$ is not a finite union of convex polyhedra.
In particular, $M$ is not a tropical prevariety.
\end{proposition}

\begin{proof}
Choose an integer $2 \le j < \infty$ and consider the point
$${\bf p}_j:= (0,0, -j/2- 1/4 + \eps_j, -j/2- 1/4, -j,-j) \in \Real^6.$$
A direct calculation shows that ${\bf p}_j \in M$ because minima in the set ${\mathcal L}_i({\bf p}_j)$
are attained at
\begin{itemize}
\item
coordinates $x_1,x_2$ if $i>j$,
\item
coordinates $y_1,y_2$ if $i=j$,
\item
coordinates $z_1,z_2$ if $i<j$.
\end{itemize}
Moreover, the same calculation shows that any point
${\bf p}_j + (\delta_1,\delta_1, \delta_2, \delta_2, \delta_3, \delta_3)$, for all
sufficiently small positive $| \delta_1|, | \delta_2|, | \delta_3|$, also belongs to $M$.
Hence, a neighbourhood of ${\bf p}_j$ in $M$ contains a 3-cube, denote it by $C_j$.

Conversely, a neighbourhood of ${\bf p}_j$ in $M$ is contained in $C_j$.
Indeed, consider a point ${\bf q}_j \in M$ which is sufficiently close to ${\bf p}_j$ in $M$.
It can be represented as
$${\bf q}_j={\bf p}_j+ (\alpha_1, \ldots, \alpha_6)=$$
$$=(\alpha_1, \alpha_2, -j/2 -1/4 +\eps_j+ \alpha_3, -j/2- 1/4+ \alpha_4, -j+ \alpha_5, -j+ \alpha_6)$$
for some small positive $| \alpha_\ell |$, $1 \le \ell \le 6$.
Since ${\bf q}_j \in L_{j-1} \cap L_j \cap L_{j+1}$, we conclude that each of the following three sets
has at least two minimal elements:
\begin{enumerate}
\item
$\{ 1+\alpha_1, 1+ \alpha_2, 1/4+ \alpha_3, 1/4 + \alpha_4, \alpha_5, \alpha_6 \},$
\item
$\{ \alpha_1, \alpha_2, -1/4 +\alpha_3, -1/4 + \alpha_4, \alpha_5, \alpha_6 \},$
\item
$\{ -1+ \alpha_1, -1+ \alpha_2, -3/4+ \alpha_3, -3/4 + \alpha_4, \alpha_5, \alpha_6 \}.$
\end{enumerate}
Since all $|\alpha_\ell|$ are small, minimal (hence equal) elements in (1) are $\alpha_5, \alpha_6$,
in (2) they are $-1/4 +\alpha_3, -1/4 + \alpha_4$, and in (3) they are $-1+ \alpha_1, -1+ \alpha_2$.
It follows that $\alpha_1= \alpha_2$, $\alpha_3= \alpha_4$, and $\alpha_5= \alpha_6$, hence
${\bf q}_j \in C_j$.

We have proved that the 3-cube $C_j$ is a neighbourhood of ${\bf p}_j$ in $M$.
By a direct calculation, any two points of the kind
$${\bf p}_i + (a_1,a_1,a_2,a_2,a_3,a_3)\> \text{and}\> {\bf p}_j+ (b_1,b_1,b_2,b_2,b_3,b_3),$$
where $i \neq j$ and all $a_\ell, b_\ell \neq \infty$, $1 \le \ell \le 3$, are distinct due to $\eps_i \neq \eps_j$.
This means that affine hulls of cubes $C_i$ and $C_j$ are disjoint in $\Real^6$ for
any two $i \neq j$, $2 \le i,j < \infty$, and therefore, $M \cap \Real^6$ is not a finite union of convex polyhedra.
Since every tropical prevariety is a finite polyhedral complex (see, e.g., \cite{MS}),
$M$ is not a tropical prevariety.
\end{proof}

\section*{Acknowledgements}
We thank M. Joswig, N. Kalinin, H. Markwig, and T. Theobald for useful discussions, and anonymous referees for
constructive remarks and suggestions.
Part of this research was carried out during our joint visit in September 2017 to the Hausdorff
Research Institute for Mathematics at Bonn University, under the program Applied and Computational Algebraic
Topology, to which we are very grateful.
D. Grigoriev was partly supported by the RSF grant 16-11-10075.

\end{document}